\documentclass{article}

\usepackage{arxiv}

\usepackage[utf8]{inputenc} % allow utf-8 input
\usepackage[T1]{fontenc}    % use 8-bit T1 fonts
\usepackage{hyperref}       % hyperlinks
\usepackage{url}            % simple URL typesetting
\usepackage{booktabs}       % professional-quality tables
\usepackage{amsfonts}       % blackboard math symbols
\usepackage{nicefrac}       % compact symbols for 1/2, etc.
\usepackage{microtype}      % microtypography
\usepackage{lipsum}		% Can be removed after putting your text content
\usepackage{graphicx}
\usepackage{natbib}
\usepackage{doi}
\usepackage{arxiv}

\usepackage[utf8]{inputenc} % allow utf-8 input
\usepackage[T1]{fontenc}    % use 8-bit T1 fonts
\usepackage{hyperref}       % hyperlinks
\usepackage{url}            % simple URL typesetting
\usepackage{booktabs}       % professional-quality tables
\usepackage{amsfonts}       % blackboard math symbols
\usepackage{nicefrac}       % compact symbols for 1/2, etc.
\usepackage{microtype}      % microtypography
\usepackage{lipsum}		    % Can be removed after putting your text content
\usepackage{graphicx}
\usepackage{natbib}
\usepackage{doi}

\usepackage{appendix}

\usepackage{xcolor}
\RequirePackage{amsthm,amsmath,amsfonts,amssymb}
\RequirePackage{natbib}
\newcommand{\var}{\mathrm{Var}}

\newtheoremstyle{customprop}
  {\topsep}   % Space above
  {\topsep}   % Space below
  {\itshape}  % Body font
  {}          % Indent amount
  {\bfseries} % Theorem head font
  {.}         % Punctuation after theorem head
  {.5em}      % Space after theorem head
  {}          % Theorem head spec

\theoremstyle{customprop}
\newtheorem{App_prop}{Proposition}

\RequirePackage{bbold}

\usepackage{comment}
\usepackage[inline]{enumitem}
\theoremstyle{plain}

\newtheorem{theorem}{Theorem}

\newtheorem{algo}{Algorithm}%[section]
\newtheorem{proposition}{Proposition}%[section]
\theoremstyle{definition}
\newtheorem{definition}{Definition}

\newtheorem{assumption}{Assumption}

\title{A template for the \emph{arxiv} style}

\title{Bootstrap Validity in Bayesian Semi-Parametric Models  }
\renewcommand{\shorttitle}{Bootstrap validity in Bayesian semi-parametric models}

\author{ \hspace{1mm}Magid Sabbagh\\
	Department of Mathematics and Statistics\\
	McGill University\\
	Montreal, QC, Canada \\
	\texttt{magid.sabbagh@mail.mcgill.ca} \\
	\And
	\href{https://orcid.org/0000-0001-9811-7140}{\includegraphics[scale=0.06]{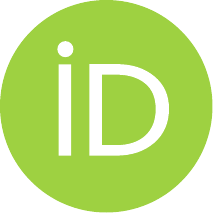}\hspace{1mm}David A. Stephens}\thanks{Corresponding author} \\
	Department of Mathematics and Statistics\\
	McGill University\\
	Montreal, QC, Canada \\
	\texttt{david.stephens@mcgill.ca} \\
}

\renewcommand{\headeright}{Technical Report}
\hypersetup{
pdftitle={A template for the arxiv style},
pdfsubject={q-bio.NC, q-bio.QM},
pdfauthor={David S.~Hippocampus, Elias D.~Striatum},
pdfkeywords={First keyword, Second keyword, More},
}

\begin{document}
\maketitle

\begin{abstract}
We discuss Bayesian inference on a low-dimensional targeted parameter in the presence of possibly highly complex nuisance components within the semi-parametric inference framework using an estimating function approach. We obtain a posterior distribution using non-parametric Bayesian methods through the Dirichlet process and the Bayesian bootstrap. We relax the commonly deployed notion of stochastic equicontinuity and develop a framework leading to posterior inference with good frequentist properties, specifically we demonstrate that the posterior distribution is asymptotically Normal and concentrates at the true value of the parameter. We emphasize the specific assumptions that are required to obtain these results, and how relaxing any of them alters the conclusions.  We verify the analytical results in simulation.

%\textcolor{red}{The abstract doesn't really describe what is in the paper - you never mention likelihood times prior, stochastic equicontinuity : or causal inference !}
\end{abstract}

% keywords can be removed

\section{Introduction}
We suppose that a targeted parameter  $\theta_0 \in \Theta \subseteq \mathbb{R}^p$ and a nuisance parameter $h_0 \in \mathcal{H}$ satisfy the moment restriction $\mathbb{E}m(O;\theta_0,h_0)=0,$ where $O$ has distribution $P_O$. If $O_1,\ldots,O_N$ are i.i.d. with distribution $P_0$ and $h_0$ is known, it is possible to obtain a valid posterior distribution for $\theta_0$ by solving the weighted estimating equation
\begin{equation}\label{eq:BBmoment}
\sum_{i=1}^{N} w_{iN}m(O_i;\theta,h_0)=0,
\end{equation}
with weights drawn from the model $(w_{1N},\ldots,w_{NN}) \sim \text{Dir}(1,\ldots,1)$ independent of $O_1,\ldots,O_N$, which is known as the Bayesian bootstrap; see \citet{Rubin:1981}, \citet{Ghosal/VanderVaart:2017} and \citet{Lyddon/etal:2019}. The Bayesian justification for this procedure comes from the fact that the targeted parameter is a functional of $P_0$, and therefore the posterior distribution for $\theta_0$ is merely the probability distribution for the functional derived directly from the posterior for $P_0$; the Bayesian bootstrap arises as a limiting case of the standard nonparametric inference for $P_0$ based on the Dirichlet process.  Specifically, for the moment constraint
\[
Em(O;\theta_0,h_0)\equiv \int m(o;\theta_0,h_0) dP_0(o) = 0
\]
the posterior uncertainty for $\theta_0$ can be obtained from the Dirichlet process posterior for $P_0$.  A Dirichlet process prior with prior measure hyperparameter $\alpha_0$, $\mathcal{DP}(\alpha_0)$, yields a posterior which is $\mathcal{DP}(\alpha_N)$; this is a distribution on the space of discrete distributions, and random distributions drawn from $\mathcal{DP}(\alpha_N)$ have the stochastic representation
\[
V_N Q_0 + (1-V_N) P_N^w
\]
%\textcolor{blue}{Should it be $V_n \mathcal{DP}(Q_0)+(1-V_n)P_n^{w}$ as $P_n^{w}$ is a process}
where
\begin{itemize}
    \item $V_N \sim Beta(|\alpha_0|,N)$
    \item $Q_0$ is a random discrete measure drawn from $\mathcal{DP}(\alpha_0)$
    \item $P_N^w$ is the random discrete measure concentrating on the observed $o_1,\ldots,o_N$ with random exchangeable weights $(w_{1N},\ldots,w_{NN})$.
\end{itemize}

Therefore to obtain a posterior sample for $\theta_0$ (assuming $h_0$ is known), we solve
\begin{equation}\label{eq:DPmoment}
v_N \sum_{j=1}^\infty w_j^\ast m(o_j^\ast;\theta,h_0) + (1-v_N) \sum_{i=1}^N w_{iN} m(o_i;\theta,h_0) = 0
\end{equation}
where $v_N$ is a draw from $Beta(|\alpha_0|,N)$ and $\{(w_j^\ast,o_j^\ast),j=1,2,\ldots\}$ constitute a draw from prior measure $\mathcal{DP}(\alpha_0)$ that can be obtained using standard algorithms such as stick-breaking, and $(w_{1N},\ldots,w_{NN})$ are Dirichlet weights as defined above.   Under the Dirichlet process assumptions, this method produces an exact sample from the posterior for $\theta_0$, and therefore this is a fully Bayesian (Monte Carlo) approach to inference.

Despite the fact that we have an exact calculation for finite $N$, much of our focus in this paper will be on frequentist properties of posterior when $N$ is large, and in the limit as $N \rightarrow \infty$.  When $|\alpha_0|$ is small, the infinite sum can be truncated as the stick-breaking weights quickly become zero to machine accuracy. Further, from \eqref{eq:DPmoment}, it is evident that under mild conditions on $m$, the first sum is negligible as $N \longrightarrow \infty$ (as the the posterior concentrates on the observed data) and so our asymptotic results may be derived using \eqref{eq:BBmoment} rather than the more general \eqref{eq:DPmoment}.  In many cases related to this calculation, we will be able to note a Bayesian/frequentist duality, in that Bayesian posterior is asymptotically identified by the frequentist sampling distribution.  We contend that the Bayesian interpretation is in all cases preferable as it delivers the desired conditional probability statement for the targeted unknown parameter, rather than an unconditional (repeated sampling) statement that does not facilitate any meaningful representation of uncertainty based on the actual data observed.

Against this background, this paper studies the problem of making inference about $\theta_0$ using the Dirichlet process posterior/Bayesian bootstrap when $h_0$ is unknown and also needs to be estimated.  In particular, we study the behaviour of the resulting posterior under minimal conditions. We focus on the Bayesian bootstrap for illustration, but note that the weights for which the presented results hold encompass a more general family than the standard Dirichlet weights; see \citet{Praestgaard/Wellner:1993}, \citet{Sabbagh/Stephens:2026,SSb2026} for more details.  In Section \ref{sec:BC}, we state the core properties of frequentist and Bayesian semi-parametric strategies that are facilitated by general results on the bootstrap in the case where the nuisance parameter is known, or estimated using a parametric model, or subject to certain simplifying assumptions.  Section \ref{sec:results} tackles the case when the nuisance parameter is highly complex where standard simplifying assumptions do not apply. Specifically, we drop the restrictions on the complexity of the nuisance space and use the cross-fitting approach and algorithms from \citet{Chernozhukov/etal:2018} in order to relate a Bayesian posterior distribution computed using the Bayesian bootstrap estimator to properties of a frequentist counterpart.  We provide results establishing the required asymptotic behaviour.  Section \ref{sec:Sims} verifies the relevance of the theoretical results in finite sample using simulations in the partially linear model.

\section{Bootstrap consistency under simplifying assumptions}
\label{sec:BC}

When $h_0$ is known, the following result \citep{Kosorok:2008} enables us to assert the validity of the bootstrap for such procedures (specifically, consistency and asymptotic normality of the frequentist and Bayesian quantities), which we term \textit{bootstrap consistency}:
\begin{theorem}\label{thmllbroot}
    Let $\hat{\theta}_n$ and $\hat{\theta}_{n,BB}$ satisfy 
    \[
    \sum\limits_{i=1}^n m(O_i;\hat{\theta}_n,h_0)=0 \qquad \text{and} \qquad  \sum\limits_{i=1}^n w_{in} m(O_i;\hat{\theta}_{n,BB},h_0)=0
    \]
    respectively.  Then, under regularity conditions,
\[
   \sqrt{n} ( \hat{\theta}_n-\theta_0 ) 
 \qquad \text{and} \qquad   \sqrt{n} (\hat{\theta}_{n,BB}-\hat{\theta}_n)|O_1,\ldots O_n
   \]
   converge in distribution as $n \to \infty$ to random variables having a $\mathcal{N}(0, \Sigma)$ distribution, where 
   \[
   \Sigma= \left[\mathbb{E}_{P_O}\left\{ \frac{\partial{m(O;\theta_0,h_0)}}{\partial \theta^\top}\right\}\right]^{-1} \mathbb{E}_{P_O}\left\{ m(O;\theta_0,h_0)m(O;\theta_0,h_0)^\top \right\}\left[\mathbb{E}_{P_O}\left\{ \frac{\partial{m(O;\theta_0,h_0)}}{\partial \theta^\top}\right\}\right]^{-\top}
   \]
\end{theorem}

Thus when $h_0$ is known, inference proceeds in a standard fashion, and the theorem confirms the desirable properties from both viewpoints. 

When the nuisance parameter $h_0$ needs to be estimated, and estimation is performed using a parametric model independent of $\theta$, inference for $\theta$ proceeds under mild conditions in a relatively straightforward fashion.  \citet{Stephens/etal:2022} introduce the Linked Bayesian Bootstrap algorithm that achieves bootstrap consistency; see \citet{Sabbagh/Stephens:2026} for technical details concerning the Linked Bayesian Bootstrap and its consequences. When $h_0$ is modeled non-parametrically but is known to lie in a low-complexity space (specifically, Donsker classes), \citet{SSb2026} establish this duality extending arguments of \citet{Newey:1994} for Neyman orthogonal scores (see \cite{Chernozhukov/etal:2018} and Section \ref{sec:results}). \citet{SSb2026} also notes that posterior convergence can be achieved without the score being Neyman orthogonal. In fact, in many interesting examples, it can be shown that it is the properties of the frequentist estimator that dictate the Bayesian/frequentist duality as the posterior distribution is well-behaved.

  %%% Uncomment this line and comment out the ``thebibliography'' section below to use the external .bib file (using bibtex) .





% \end{thebibliography}

\section{Bootstrap consistency with an unknown nuisance parameter}
\label{sec:results}

%\subsection{Assumptions}
In the remaining parts of the paper, we suppose that the targeted parameter space $\Theta$ is a subset of $\mathbb{R}^p$ and that the nuisance space  $\mathcal{H}$ is a convex subset of a normed space. For a vector $v \in \mathbb{R}^p$, we write $\Vert v \Vert_{p,2}$ for the Euclidean norm of $v$ and for $h \in \mathcal{H},$ we write $ \Vert h \Vert_{\mathcal{H}}$ for the norm of $h$. The true values $\theta_0 \in \Theta$ of  $\theta $ and $h_0 \in \mathcal{H}$ of $h$ satisfy the moment restriction $E_{P_O}\left\{m(O;\theta_0,h_0)\right\}=0$, for some score function $m : \mathcal{O} \times \Theta \times \mathcal{H} \to \mathbb{R}^p$.

\subsection{Neyman orthogonality} 
\label{sec:NO}
A key concept in the derivation of semi-parametric estimators and related inference is that of orthogonality (or \textit{Neyman orthogonality}) of score functions:
\begin{definition}{\label{NeymanOrth}}
For estimating/score function $m$, let $f_{h}(t)=E_{P_O} m\left\{O;\theta_0,h_0+t(h-h_0)\right\}$ for all $h \in \mathcal{H}$.
       Then $m$ is \textit{orthogonal} with respect to $h$ if $f_{h}^\prime(0)=0$, 
        where $f^\prime$ denotes the first derivative of $f$ with respect to its argument.
\end{definition}

\begin{comment}
\textcolor{blue}{In the example, $h_0$ has two components $(k_0,e_0)$ -- how do we define $t$ in that case ?  Is $t$ still scalar ?  Do we need partial derivatives ?}
\textcolor{red}{Yes, $t$ is still a scalar. We have examples in \citet{SSb2026} on how to calculate this.}

\textcolor{blue}{So maybe that needs to be noted here ?}
\textcolor{red}{Yes, sure. We can maybe refer to an example, and show the calculation for the simulation ? And here we can add that $f:[0,1] \to \mathbb{R}^p$}
\end{comment}

The concept of Neyman orthogonality is of crucial importance in econometrics and statistics. \citet{Newey:1994} uses this concept in order to limit the effect of non-parametrically estimating the nuisance parameter in carrying frequentist inference of the parameter of interest. Given a score, it is rather straightforward to check whether it satisfies the Neyman orthogonality property by simply verifying that the property in Definition \ref{NeymanOrth} holds. However, constructing Neyman orthogonal scores from non-orthogonal scores can be more complex. \citet{Chernozhukov/etal:2018}  show how to obtain Neyman orthogonal scores in different cases, whereas \citet{Kennedy:2023} relates the Neyman orthogonal score to the efficient influence functions of a parameter $\theta_0 = -\mathbb{E}B(O;h_0)$. 

In this setting, assuming orthogonality of the score function used for estimation, we seek an approach that yields a well-behaved posterior distribution under minimal conditions on $h_0$ and its estimation, such as might arise when using flexible methods, or if the data lie in high dimension. Our approach is based on sample-splitting and cross-fitting. We establish a duality between the (frequentist) sampling distribution computed using sample-splitting and the (Bayesian) posterior distribution of the targeted parameter, and in order to do that we must extend the sample-splitting technique to be used with bootstrap weights;  we provide the necessary theoretical extensions.

\subsection{Relaxation of stochastic equicontinuity} It is useful to elaborate on the difference between the usefulness of the stochastic equicontinuity property of estimators belonging to a Donsker class and the method of sample-splitting and cross-fitting. Suppose that $\hat{f}$ is an estimator of $f$ based on an i.i.d sample $O_1,\ldots,O_n$  and suppose that we wish to study the properties of
 \begin{equation}\label{eq:Dterm}
 \sqrt{n} \displaystyle{\left\{ \frac{1}{n} \sum\limits_{i=1}^n \hat{f}(O_i)-\mathbb{E}\hat{f}(O) \right\}},
 \end{equation}
 where $\hat{f}$ belongs to a Donsker class (with probability 1, or high probability). The stochastic equicontinuity property tells us that if $\hat{f}$ is close enough to $f$ (so that the variance of $\hat{f}(O)-f(O)$ goes to 0 in probability), 
 then
 \[
 \sqrt{n} \displaystyle{\left\{ \frac{1}{n}\sum\limits_{i=1}^n \hat{f}(O_i)-\mathbb{E}\hat{f}(O) \right\}}- \sqrt{n} \displaystyle{\left\{ \frac{1}{n}\sum\limits_{i=1}^n {f}(O_i)-\mathbb{E}{f}(O) \right\}} \to 0
 \]     
in probability as $n \to \infty$. This difference, which happens to be  asymptotically negligible, is omnipresent in the literature and helps simplifying the theoretical arguments. In fact, without the Donsker property, terms such as \eqref{eq:Dterm} may be difficult to study because of the dependence present there. One can interpret the Donsker assumptions as assumptions on the smoothness of the estimator and the complexity of the space in which it lies. We refer to \citet{Chernozhukov/etal:2018}, \citet{Kennedy:2016}, and \citet{VanderVaart:1996} for examples on Donsker classes and their properties. 

Whenever the Donsker property fails to hold, one may resort to sample-splitting and cross-fitting to establish the properties of terms that are studied using stochastic equicontinuity. The benefit of sample-splitting as a procedure is that it is theoretically simple to study with reasonably minor conditions. Rather than imposing complexity restrictions on the nuisance parameter space, sample-splitting requires  consistency of the nuisance estimator with respect to a criterion, the root mean-squared error for example. \citet{Zheng/vanderLaan:2010}, \citet{Robins/etal:2008} and \citet{Chernozhukov/etal:2018} have used sample-splitting in different contexts, targeted MLE, influence functions and causal estimation, respectively. Their main goal is to overcome the bias and overfitting induced by fitting highly complex nuisance estimators by splitting the data into training and testing sets. The mathematical tools used in sample-splitting turn out to be less sophisticated than the relevant notions in empirical processes and consist mainly in conditioning on the training set and exploiting the independence of the training and testing sets. This conditioning renders the estimator obtained through the training set and plugged-in in the testing set non-stochastic.

\subsection{Sample-splitting with bootstrap resampling}{\label{SSalgo}} 

Suppose that $O_1,\ldots,O_{2n}$ are i.i.d random variables and assume for simplicity that $I_1=\left\{1,\ldots,n \right\}$ and $I_2=\left\{n+1, \ldots,2n \right\}$. First, we train the model on $I_2$ to obtain an estimator $\hat{f}$ of a function $f$ and then consider $\hat{f}(O_i)$ for $i \in I_1$. The sum of such terms in \eqref{eq:Dterm} does not consist of i.i.d terms, but conditioning on $I_2$ renders the terms conditionally i.i.d. terms. Conditioning on $I_2$ and using the laws of total expectation and variance will enable us to study terms of the form
\[
\sqrt{n} \displaystyle{\left\{ \frac{1}{n}\sum\limits_{i=1}^n \hat{f}(O_i)-\mathbb{E}\hat{f}(O) \right\}}-\sqrt{n} \displaystyle{\left\{ \frac{1}{n}\sum\limits_{i=1}^n {f}(X_i)-\mathbb{E}{f}(X_i) \right\}},
\]
 and show that they tend to zero in probability. 
 
The basic sample-splitting algorithm for estimation in the presence of a nuisance parameter \citep{Chernozhukov/etal:2018} proceeds as follows. Let $\hat{h}_2$ be an estimator of $h_0$ based on $I_2$ and let 
 $\hat{\theta}_{1,n}$ and $\hat{\theta}_{1,n,BB}$ satisfy the following moment conditions
\begin{align*}
\frac{1}{n}\sum\limits_{i\in I_1} m(O_i;\hat{\theta}_{1,n},\hat{h}_2)&\equiv\frac{1}{n}\sum\limits_{i=1}^n m(O_i;\hat{\theta}_{1,n},\hat{h}_2)=0
\end{align*}
and for the Bayesian bootstrap
\begin{align*}
\sum\limits_{i \in I_1} w_{in}m(O_i;\hat{\theta}_{1,n,BB},\hat{h}_2)
&\equiv \sum\limits_{i=1}^n w_{in}m(O_i;\hat{\theta}_{1,n,BB},\hat{h}_2)=0.
\end{align*}
Similarly, let $\hat{h}_1$ be an estimator of $h_0$ based on $I_1$ and let 
 $\hat{\theta}_{2,n}$ and $\hat{\theta}_{2,n,BB}$ satisfy the following 
\begin{align*}
\frac{1}{n}\sum\limits_{i \in I_2}m(O_i,\hat{\theta}_{2,n},\hat{h}_1) &\equiv\frac{1}{n}\sum\limits_{i=1}^{n} m(O_{n+i};\hat{\theta}_{2,n},\hat{h}_1)=0
\end{align*}
and 
\begin{align*}
\sum\limits_{i \in I_2}w_{in}m(O_i,\hat{\theta}_{2,n,BB},\hat{h}_1)&\equiv\sum\limits_{i=1}^{n} w_{in}m(O_{n+i};\hat{\theta}_{2,n,BB},\hat{h}_1)=0.
\end{align*}
We construct estimators
$\hat{\theta}_{n}$ and $\hat{\theta}_{n,BB}$ of $\theta_0$ as
\[
\hat{\theta}_n=\frac{1}{2} \left(\hat{\theta}_{1,n}+\hat{\theta}_{2,n} \right)
\qquad \text{and} \qquad
\hat{\theta}_{n,BB}=\frac{1}{2} \left(\hat{\theta}_{1,n,BB}+\hat{\theta}_{2,n,BB}\right).
\]
These estimators, and similar estimators derived using more versions of the algorithm, will be used for inference. Our aim is to show that the distributions of 
\[
\sqrt{N}\left(\hat{\theta}_n-\theta_0\right)
\quad \text{ and } \quad 
\sqrt{N}\left(\hat{\theta}_{n,BB}-\hat{\theta}_n\right) \mid O_1,\ldots,O_N 
\]
are identical.  Note that the first of these quantities is a frequentist quantity whose randomness arises from $\hat{\theta}_n$, whereas the second is a conditional (Bayesian) quantity with $\hat{\theta}_n$ fixed by the conditioning, whose randomness arises due to the bootstrap weights.

The method of estimation used for $h_0$ is not specified here, but estimation could be implemented using Bayesian (such as BART) or frequentist methods; our results are independent of the method of estimation of the nuisance parameter, provided certain regularity conditions are met, as described in Section \ref{SSResults}. The focus of this report is the relaxation of the stochastic equicontinuity/Donsker assumptions presented in \citet{SSb2026}.

\subsection{Results for Linear Scores}{\label{SSResults}}
Suppose that $m$ is a linear score in $\theta$ with
\[
m(O;\theta,h)=A(O;h)\theta+B(O;h)
\]
where $A: \mathcal{O} \times \mathcal{H} \mapsto \mathbb{R}^{p \times p}$ and $B: \mathcal{O} \times \mathcal{H} \mapsto \mathbb{R}^{p}$.   This special case is of practical interest: for example, in causal inference many scores satisfied by causal quantities are linear scores. For example, the doubly robust estimators of the average treatment effect (ATE) and the average treatment effect on the treated (ATT) can be derived from Neyman orthogonal scores that are linear in the parameter of interest. We suppose that $\hat{h}_1$ and $\hat{h}_2$, defined previously in Section \ref{SSalgo}, consistently estimate $h_0$, that is, they belong to $H_N$, where $\left\{H_i\right\}_{i=1}^\infty$ are shrinking neighborhoods around $h_0$.

We establish the properties of $\sqrt{n}(\hat{\theta}_{1,n,BB}-{\theta}_{0})$, $\sqrt{n} (\hat{\theta}_{1,n}-{\theta}_{0})$ and of $\sqrt{n}(\hat{\theta}_{1,n,BB}-\hat{\theta}_{n})$. We use the following expansions: write
\begin{align*}
    0&=\sum\limits_{i=1}^n m(O_i,\theta_0,\hat{h}_2)+\left\{\sum\limits_{i=1}^n A(O_i,\hat{h}_2)\right\} \left(\hat{\theta}_{1,n}-{\theta}_{0}\right)\\
   0&=\sum\limits_{i=1}^n w_{in}m(O_i,\theta_0,\hat{h}_2)+\left\{\sum\limits_{i=1}^n w_{in}A(O_i,\hat{h}_2)\right\} \left(\hat{\theta}_{1,n,BB}-{\theta}_{0}\right).
\end{align*}
%\textcolor{red}{Is there something missing here ?: Added!}

We further expand 
\begin{align*}
\sqrt{n}\sum\limits_{i=1}^n & w_{in} m(O_i;\theta_0,\hat{h}_2)\\
& =\sqrt{n}\sum\limits_{i=1}^n w_{in}m(O_i;\theta_0,h_0)
+\sqrt{n} \sum\limits_{i=1}^n \left(w_{in}-\frac{1}{n}\right)
\left\{m\left(O_i;\theta_0,\hat{h}_2\right)-m\left(O_i;\theta_0,h_0\right)\right\}\\
&\ + \sqrt{n}\left[ \frac{1}{n}\sum\limits_{i=1}^n\left\{m(O_i;\theta_0,\hat{h}_2)-m(O_i;\theta_0,h_0)\right\}-\mathbb{E}_{P_O}\left\{m\left(O;\theta_0,\hat{h}_2\right)-m(O;\theta_0,h_0)\right\}\right]\\
&\ + \sqrt{n}\mathbb{E}_{P_O}\left\{m\left(O;\theta_0,\hat{h}_2\right)\right\}
\end{align*}
By setting the weights $w_{in}=1/n$ we obtain an expansion of $n^{-1/2}\sum\limits_{i=1}^n m(O_i;\theta_0,\hat{h}_2)$, an expansion used by \citet{Andrews:1994a} and \citet{Chernozhukov/etal:2018}.

We make the following assumptions and explain how each of these assumptions is used in establishing the main results. We suppose that the $p \times p$ matrix $M_{\theta_0}:=\mathbb{E} A(O;h_0)$ is invertible. Technical details and modified conditions can be found in the supplementary material.
 \begin{assumption}{\label{AssonA1}}
         $\sup\limits_{h \in H_N}  \left\Vert  \mathbb{E} A\left(O;h\right)-\mathbb{E}A\left(O;h_0\right) \right\Vert_{p,2} \leq a_N$, where $(a_N)_{N=1}^{\infty}$ is a sequence of positive numbers converging to zero.
\end{assumption}
\begin{assumption}{\label{AssonA2}}
         $\sup\limits_{h \in H_N}   \mathbb{E} \left\{\Vert A(O;h) \Vert_{p,2}^2\right\}  \leq C$ for some positive constant $C$.
\end{assumption}
\begin{assumption}{\label{AssonVarm}}
         $\displaystyle{\sup\limits_{h \in H_N} \mathbb{E}\left\{\Vert m\left(O;\theta_0,h\right) -m\left(O;\theta_0,h_0\right)\Vert_{p,2}^2\right\} \leq \delta_N^2}$, where $(\delta_N)_{N=1}^{\infty}$ is a sequence of positive numbers converging to zero.
\end{assumption}
\begin{assumption}{\label{Assonf}}
  Let  $f_h(t)= \mathbb{E} m\left\{O;\theta_0, h_0+t(h-h_0)\right\}$, for $h \in H_N $.
 We suppose that $f$ is twice differentiable on $(0,1)$ and that $f_h^\prime(0)=0$ and \[
 \displaystyle{ \sqrt{n}\int_0^1  \Vert f_h^{\prime \prime}(t) \Vert } dt \to 0\]
 as $n \to \infty$, uniformly in $h \in H_N$.  This implies that
 \[
 \sqrt{n} \mathbb{E}_{P_O}m(O;\theta_0,h) \longrightarrow 0
 \]
 uniformly in $h$. 
\end{assumption}

\begin{comment}
\textcolor{blue}{It seems like 3. and 4. are sort of continuity assumptions, is that correct ?  Essentially, are they really weaker than Donsker-type assumptions ?}
\textcolor{red}{ A3: Yes indeed is weaker. Donsker + these assumptions imply that we can use stochastic equicontinuity. Here we are using A3+ Sample Splitting. They depend more on the properites of $m$ and the way $h$ is estimated. \citet{Chernozhukov/etal:2018} has many examples at the end. A4 is more a second derivative bound. No stochastic equicontinuity property is required. Used in \citet{Newey:1994} and \citet{Chernozhukov/etal:2018}}

\textcolor{blue}{Yes, of course, I understand all that: my question is whether it really is weaker, or just different.   I see that there is no explicit Donsker requirement here; but the thing that concerns me is this requirement for uniform 'regular' behaviour on a shrinking $H_N$ -- are we sure that there is not something hidden in there ?  For example, in A3: this assumption means that for the estimator of h the variation is bounded by something that goes to zero; so not Donsker, but still quite strong, no ?}
\end{comment}

\noindent \textbf{Notes:}
\begin{enumerate}
  \item  Assumptions \ref{AssonA1} and \ref{AssonA2} enable us to establish that 
  \[
  \sum\limits_{i=1}^n w_{in}A(O_i;\hat{h}_2) \quad \text{ and } \quad \frac{1}{n}\sum\limits_{i=1}^nA(O_i,\hat{h}_2) 
  \]
  converge unconditionally to $\mathbb{E}A(O;h_0)$. In general, the use of sample-splitting to ensure the convergence in probability of these terms is not necessary when stronger yet reasonable conditions are imposed. 
  \item  Assumption \ref{AssonVarm} along with sample-splitting enables us to establish that the empirical process term 
   \[
\sqrt{n} \left[\frac{1}{n}\sum\limits_{i=1}^n\left\{m(O_i;\theta_0,\hat{h}_2)-m(O_i;\theta_0,h_0)\right\}-\mathbb{E}_{P_O}\left\{m\left(O;\theta_0,\hat{h}_2\right)-m(O;\theta_0,h_0)\right\}\right]
\]
and bootstrapped empirical process term
\[
\sqrt{n} \sum\limits_{i=1}^n \left(w_{in}-\frac{1}{n}\right)
\left\{m\left(O_i;\theta_0,\hat{h}_2\right)-m\left(O_i;\theta_0,h_0\right)\right\},
\]
  converge unconditionally to zero in probability. Under the classical Donsker assumption, the result is obtained by using the stochastic equicontinuity properties and invoking the properties of the exchangeable bootstrap. The proof involves using the law of total variance by conditioning on $I_2$ and $\mathbf{W}_n$ to show that each component of the quantities studies go to zero in probability. We point out that when $\theta_0$ and hence the score $m$ are real valued, Assumption \ref{AssonVarm} implies that 
  \[\
\text{Var}\left\{m(O;\theta_0,\hat{h})-m(O;\theta_0,h_0)\right\} \longrightarrow 0
\]
in probability, which is a common assumption used in problems involving estimating equations; see for example \citet{VanderVaart:1998}.

\item Assumption \ref{Assonf} hinges upon Neyman orthogonality and the rate at which $h_0$ is estimated. In many models, it is sufficient that the nuisance parameters are estimated at a rate that is $o(N^{-1/4})$ as per the classical result. In fact, the second derivative $f''$ of $f$ depends in general on the rates at which the nuisance parameters need to be estimated. We require in general the estimation rate $\Vert \hat{h}-h_0 \Vert_{\mathcal{H}}=o_{P_{O}}(n^{-1/4})$ that is in fact met by many common estimation procedures.  There are more refined conditions that can be imposed on a case by case basis, especially when $h$ itself consists of two nuisance parameters; see \citet{Chernozhukov/etal:2018}.  Rates of convergence of classical non-parametric methods under regularity assumptions can be found in \citet{Gyorfi2002} and \citet{Tsybakov:2009} whereas \citet{Chernozhukov/etal:2018} surveys results about the rate of convergence of machine learning methods such as random forests and neural networks. Assumption \ref{Assonf} enables us to conclude that 
  \[
  \sqrt{n} \mathbb{E}_{P_O}m(O;\theta_0,\hat{h}) \longrightarrow 0 
  \]
  in probability. 
\end{enumerate}

 Combining the results together, and interchanging the roles of $I_1$ and $I_2$ we obtain the following proposition:
 \begin{proposition}{\label{Split}} Under Assumptions \ref{AssonA1}., \ref{AssonA2}., \ref{AssonVarm}., and \ref{Assonf}., the following hold: 
\begin{enumerate}
\item
\[
\sqrt{n} \left(\hat{\theta}_{1,n}-{\theta}_{0}\right)=-\sqrt{n} M_{\theta_0}^{-1}\sum\limits_{i=1}^n \frac{m\left(O_i;\theta_0,h_0\right)}{n} +o_{P_{O}}(1),
\]

\item
\[
\sqrt{n} \left(\hat{\theta}_{1,n,BB}-{\theta}_{0}\right)=-\sqrt{n} M_{\theta_0}^{-1}\sum\limits_{i=1}^n w_{in}m\left(O_i;\theta_0,h_0\right) +o_{P_{OW}}(1).
\]
\item 
\[
\sqrt{n} \left(\hat{\theta}_{1,n,BB}-\hat{\theta}_{1,n}\right)=-\sqrt{n} M_{\theta_0}^{-1}\sum\limits_{i=1}^n \left(w_{in}-\frac{1}{n}\right)m\left(O_i;\theta_0,h_0\right) +o_{P_{OW}}(1)
\]
\item 
\[
\sqrt{n} \left(\hat{\theta}_{2,n}-{\theta}_{0}\right)=-\sqrt{n} M_{\theta_0}^{-1}\sum\limits_{i=1}^n \frac{m\left(O_{n+i};\theta_0,h_0\right)}{n} +o_{P_{O}}(1),
\]
\item 
\[
\sqrt{n} \left(\hat{\theta}_{2,n,BB}-{\theta}_{0}\right)=-\sqrt{n} M_{\theta_0}^{-1}\sum\limits_{i=1}^n v_{in}m\left(O_{n+i};\theta_0,h_0\right) +o_{P_{OV}}(1).
\]
\item 
\[
\sqrt{n} \left(\hat{\theta}_{2,n,BB}-\hat{\theta}_{2,n}\right)=-\sqrt{n} M_{\theta_0}^{-1}\sum\limits_{i=1}^n \left(v_{in}-\frac{1}{n}\right)m\left(O_{n+i};\theta_0,h_0\right) +o_{P_{OV}}(1)
\]
 \end{enumerate}
 \end{proposition}
We now exploit the fact that $2\hat{\theta}_n=\hat{\theta}_{1,n}+\hat{\theta}_{2,n}$ and $2\hat{\theta}_{n,BB}=\hat{\theta}_{1,n,BB}+\hat{\theta}_{2,n,BB}$ to obtain the following proposition:
\begin{proposition}{\label{PropRALEstimatorsSampleSplit}} 
In light of the results obtained in Proposition \ref{Split}, the following hold

    \begin{enumerate}
        \item 
        \[
        \sqrt{n} \left(\hat{\theta}_{n}-{\theta}_{0}\right)=-\sqrt{n} M_{\theta_0}^{-1}\sum\limits_{i=1}^{2n} \frac{m\left(O_{i};\theta_0,h_0\right)}{2n} +o_{P_{O}}(1),
        \]
        \item 
        \[
         \sqrt{n} \left(\hat{\theta}_{n,BB}-{\theta}_{0}\right)=-\frac{\sqrt{n}}{2} M_{\theta_0}^{-1}\left\{\sum\limits_{i=1}^{n}w_{in} m\left(O_{i};\theta_0,h_0\right) +\sum\limits_{i=1}^n v_{in}m\left(O_{n+i};\theta_0,h_0\right)\right\}+o_{P_{OV}}(1)+o_{P_{OW}}(1)
        \]
        \item 
         \begin{align*}
        \sqrt{n} \left(\hat{\theta}_{n,BB}-\hat{\theta}_{n}\right)=&-\frac{\sqrt{n}}{2} M_{\theta_0}^{-1}\left\{\sum\limits_{i=1}^{n}\left(w_{in}-\frac{1}{n}\right) m\left(O_{i};\theta_0,h_0\right) +\sum\limits_{i=1}^n \left(v_{in}-\frac{1}{n}\right)m\left(O_{n+i};\theta_0,h_0\right)\right\}\\
        &+o_{P_{OV}}(1)+o_{P_{OW}}(1),
        \end{align*}
        \item If for all $j \in \left\{1,\ldots,n\right\}$, we have that $w_{jn}=v_{jn}$, we then get:
        \begin{enumerate}
        \item 
        \[
        \sqrt{n} \left(\hat{\theta}_{n,BB}-{\theta}_{0}\right)=-\sqrt{n} M_{\theta_0}^{-1}\sum\limits_{i=1}^{n}w_{in} \left\{m\left(O_{i};\theta_0,h_0\right)+m\left(O_{n+i};\theta_0,h_0\right)\right\}+o_{P_{OW}}(1)
        \]
       \item \[
\sqrt{n}\left(\hat{\theta}_{n,BB}-\hat{\theta}_n\right)=-\sqrt{n} M_{\theta_0}^{-1}\sum\limits_{i=1}^n \left(w_{in}-\frac{1}{n}\right)\frac{m\left(O_{i};\theta_0,h_0\right)+m\left(O_{n+i};\theta_0,h_0\right)}{2}+o_{P_{OW}}(1).
\]       
\end{enumerate}
\end{enumerate} 
\end{proposition}

Combining all the above cases in Proposition \ref{PropRALEstimatorsSampleSplit}, we obtain the following theorem:
\begin{theorem}
As $N \to \infty$, the following duality holds:  
    \begin{enumerate}
        \item  $\sqrt{N} (\hat{\theta}_{n}-{\theta}_{0}) \to \mathcal{N}(0,\Sigma)$
        \item $\sqrt{N}(\hat{\theta}_{n,BB}-\hat{\theta}_n) \mid O_1,\ldots,O_N \to \mathcal{N}(0,\Sigma)$
    \end{enumerate}
\end{theorem}
\subsection{Algorithm}
In light of the algorithm  presented in Section \ref{SSalgo} as well as Assumptions \ref{AssonA1}-\ref{Assonf}, and the results presented in Section \ref{SSResults}, we propose the following algorithm in order to obtain a posterior sample for the parameter of interest $\theta_0$. 
%\textcolor{blue}{This is the same score as at the top of page 4 - no need to write it out again}
Algorithm  \ref{BBSampleSplittingIndepWeights} illustrates how to obtain samples from the posterior distribution when the weights used for each sample are independent.

%\textcolor{blue}{Do we really need to have both of these algorithms written out ?  They are essentially the same apart from the weights in Step 3(a).}

\begin{algo}{\label{BBSampleSplittingIndepWeights}}

Sample splitting algorithm using independent weights:

  \begin{enumerate}
      \item  Split the data into two samples $I_1$ and $I_2$ of the same size.
      \item  Estimate $h_0$ using flexible methods using sample $I_1$ and denote the estimator $\hat{h}_1$ and estimate $h_0$ using ML methods using sample $I_2$ and denote the estimator $\hat{h}_2$.
      \item For each $j \in \left\{1,\ldots,B \right\},$
        \begin{enumerate}
        \item Draw random weights $(w^{(j)}_{1n},\ldots,w^{(j)}_{nn}) \sim \text{Dir}(1,\ldots,1)$ and $(u^{(j)}_{1n},\ldots,u^{(j)}_{nn}) \sim \text{Dir}(1,\ldots,1)$ independent from each other.
        
        \item Draw random distribution $\displaystyle{F_1^{(j)}=\sum\limits_{i \in I_1} w_{in}^{(j)}\delta_{O_i}}$ and $\displaystyle{F_2^{(j)}=\sum\limits_{i \in I_2} u_{in}^{(j)}\delta_{O_i}}$
        \item  Compute 
        $$
        \theta_1^{(j)}= -\frac{\sum\limits_{i\in I_1} w^{(j)}_{in}B(O_i;\hat{h}_2)}{\sum\limits_{i \in I_1} w^{(j)}_{in}A(O_i;\hat{h}_2)}
        \qquad
        \text{and} \qquad 
        \theta_2^{(j)}= -\frac{\sum\limits_{i\in I_2} u^{(j)}_{in}B(O_i;\hat{h}_1)}{\sum\limits_{i \in I_2} u^{(j)}_{in}A(O_i;\hat{h}_1)}
        $$
        \item Compute 
        $$
        \theta^{(j)}=\frac{\theta_1^{(j)}+\theta_2^{(j)}}{2}
        $$
        \end{enumerate}
        \item Output  $\theta^{(1)},\ldots,\theta^{(B)}$.
  \end{enumerate}  
\end{algo}

Step 3(a) can be replaced by a step that sets the weights to be identical without changing the validity of the algorithm.  Other methods for reporting the posterior based on replicate splits can be used; see Appendix \ref{sec:RepSplits}.

\subsection{Variance Estimation}{\label{varestss}}
The quantities $\sqrt{n}( \hat{\theta}_{1,n,BB}-\hat{\theta}_{1,n})$ and $\sqrt{n} (\hat{\theta}_{2,n,BB}-\hat{\theta}_{2,n})$ are conditionally and unconditionally asymptotically independent. In fact, $\sqrt{n}(\hat{\theta}_{1,n}-\theta_0)$ and $\sqrt{n}(\hat{\theta}_{2,n}-\theta_0)$ are also unconditionally independent. This fact prompts us to estimate the asymptotic posterior variance by $
\hat{\Sigma}=( \hat{\Sigma}_1+\hat{\Sigma}_2)/2$,
where for $j=1,2$
\[
\hat{\Sigma}_j = \left\{\frac{1}{n} \sum\limits_{i \in I_j } A(O_i;\hat{h}_{-j})\right\}^{-1} \left\{\frac{1}{n}\sum\limits_{i \in I_j} m(O_i;\hat{\theta}_{j},\hat{h}_{-j})m(O_i;\hat{\theta}_{j},\hat{h}_{-j})^\top\right\} \left\{\frac{1}{n} \sum\limits_{i \in I_j} A(O_i;\hat{h}_{-j})\right\}^{-\top}.
\]
This estimator is consistent for $\Sigma$ under reasonable conditions. We refer to  \citet{Newey:1994} and \citet{Chernozhukov/etal:2018} for more details.

\section{Simulation study}
\label{sec:Sims}
In this section, we empirically verify our theoretical results established in Section \ref{sec:results}. In particular, we will validate through simulations the asymptotic results established in Proposition \ref{SSResults} and verify that, under Neyman orthogonality, the posterior variance is not affected by the estimation of the estimation of nuisance parameter. We now study the following causal structural model, or the partially linear model studied in \citet{Robinson:1988}, \citet{Chernozhukov/etal:2018}, and \citet{Chen/etal:2022}. We have $N$ i.i.d.  observations generated using the following data generating mechanism:
\begin{align*}
Y &= \theta_0 Z + g_0(X) + U\\
Z\mid X& \sim \text{Bern}\left\{ e_0(X)\right\},
\end{align*}
where $\theta_0=3$, $X \in \mathbb{R}^{50}$. Predictor $X$ follows a multivariate Normal with zero mean and covariance matrix whose $(i,j)$-th entry is $0.8^{|i-j|}/4$, for $1 \leq i,j \leq 50$. $U$ is a standard Normal random variable and $Z$ is a binary treatment with values in $\left\{0,1\right\}$ such that $P(Z=1|X=x)=e_0(X)$.
We also assume that
\begin{align*}
 g_0(x)&=x_1+\exp(x_2-1)+ \vert x_3 \vert+\exp(x_4-1)+\vert x_5 x_6 \vert+x_{14}^2+x_{50}\\
e_0(x) &= 0.5  \frac{\exp(x_1+x_2+\ldots+x_{50}) }{1+\exp(x_1+x_2+\ldots+x_{50})} + 0.2.
\end{align*} 
The verification of the results of Proposition \ref{SSResults} will be carried out by examining the estimator of $\theta_0$ resulting from Robinson's partialling-out approach. We also indicate that similar results are found when estimating $\theta_0$ using the AIPW estimator.
Since the  nuisance parameter $g_0(X)$ may be difficult to estimate,  we resort to the partialling-out approach previously discussed in order to redefine the nuisance parameters. The relevant score equation hence becomes:
$$
m(O;\theta_0,h_0)= \left[Y-k_{y}^0(X)-\theta_0 \left\{Z-e_0(X)\right\}\right]\left\{Z-e_0(X)\right\},
$$
where $k_y^0(x)=\mathbb{E}[Y|X=x]=\theta_0 {e_0(x)} + g_0(x)$ and $h_0=(k_y^0,e_0).$  In this case, the doubly robust estimator is
\begin{equation}\label{eq:DRG}
\hat \theta_n = \dfrac{\sum\limits_{i=1}^n (Z_i - \hat e_0(X_i)) (Y_i - \hat k_y^0(X_i))}{\sum\limits_{i=1}^n (Z_i - \hat e_0(X_i))^2}.
\end{equation}

The simulation study uses sample sizes of $N=50, 100, 500$ and $1000$. We divide the data into $K=5$ subsamples each of size $n=N/K$. We obtain $\hat{\theta}_n$ as the estimator of the ATE as the average of five estimators, naturally extending the algorithm in Section \ref{SSalgo} to  $K=5$ splits instead of two. We estimate the two components of $h_0$ separately using random forests. We also estimate the variance by $\hat{\Sigma}$ as in Section \ref{varestss}, by using five splits instead of two, for the frequentist and Bayesian cases.   We report results for the standardized quantities
\[
T_N = \sqrt{N} \hat{\Sigma}^{-1/2} (\hat{\theta}_{n}-\theta_0 ) \qquad T_{N,BB} = \sqrt{N}\hat{\Sigma}^{-1/2}(\hat{\theta}_{n,BB}-\hat{\theta}_n)
\]
which, under the theory, should have standard Normal distributions. Note that in this case $\hat \Sigma$ is a scalar quantity.  We used 1000 replicate analyses to support the findings of Proposition \ref{SSResults}

The results we obtain show that the bias in estimating the ATE is eliminated by the use of sample-splitting when $N$ is sufficiently large, and reduced compared to non sample-splitting analyses when $N$ is smaller.

\begin{enumerate}

\item Figure \ref{fig:Boxsplit} shows boxplots of the frequentist estimates $\hat \theta_n$ obtained across replicates for $N=50, 100, 500$ and $1000$.  Sample-splitting reduces bias for the smaller sample sizes and removes it for $N=500$; the bias is non-negligible even for $N=1000$ due to the violation of the Donsker conditions and a non-negligible contribution from \eqref{eq:Dterm}, although there is likely a contribution to the bias that arises from the random forest estimators $h_0$ having relatively slow convergence rate.
\begin{figure}
    \centering
    \includegraphics[width=0.75\linewidth]{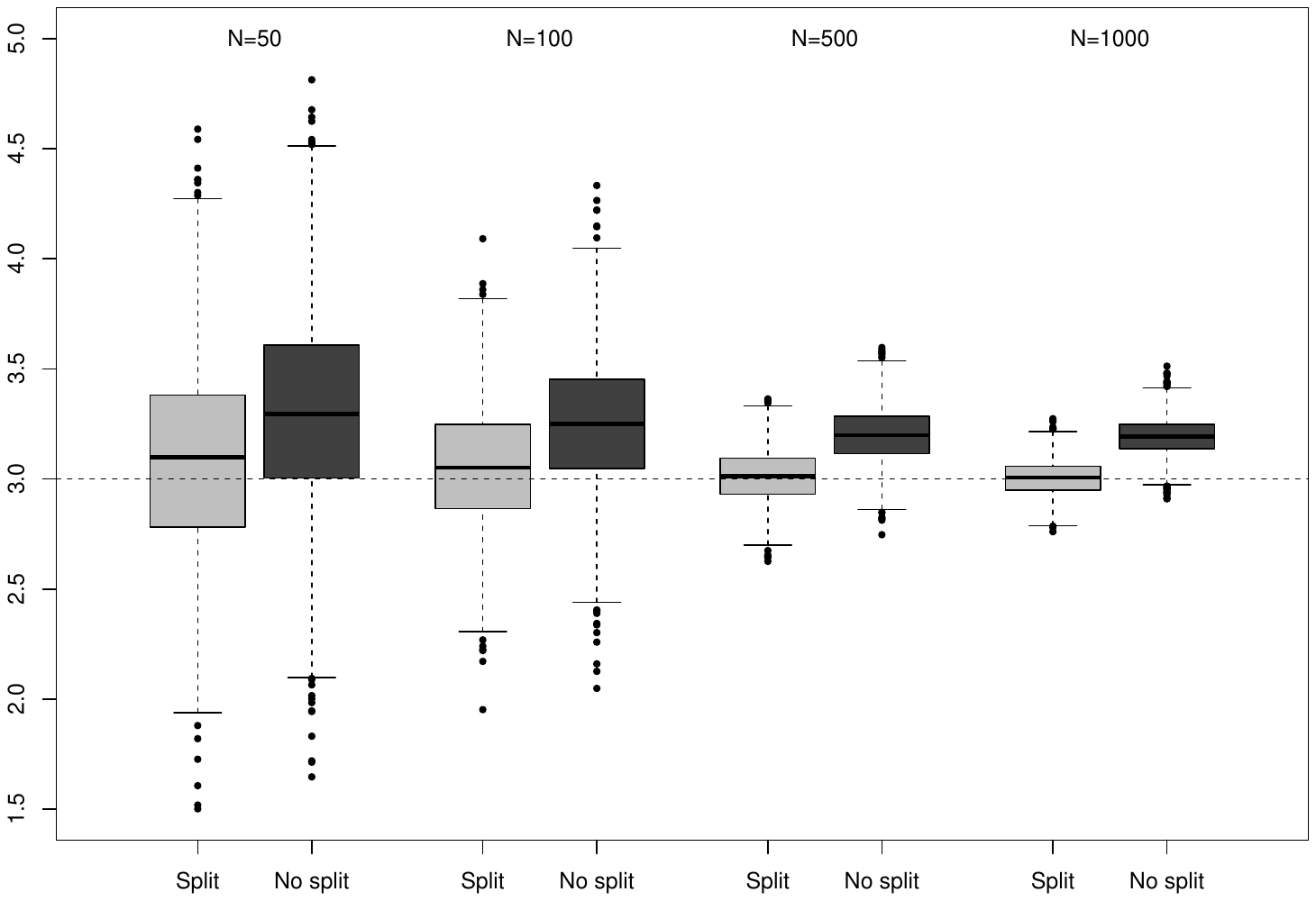}
    \caption{Boxplot of frequentist estimates for $N=50,100,500, 1000$: comparison of results for sample-splitting and no-sample splitting.}
    \label{fig:Boxsplit}
\end{figure}

    \item For $N=500$, the sampling distribution of $\hat{\theta}_{n}$ is approximately Normal, as demonstrated by the QQ plot in the second panel of Figure \ref{fig:QQ500-both}.  The first panel depicts the QQ plot for $N=50$, and illustrates that even for the smaller sample size, the Normal approximation is fairly good, even if the bias is present and the variance estimation is imperfect.
    \begin{figure}
\centering\includegraphics[width=0.48\linewidth]{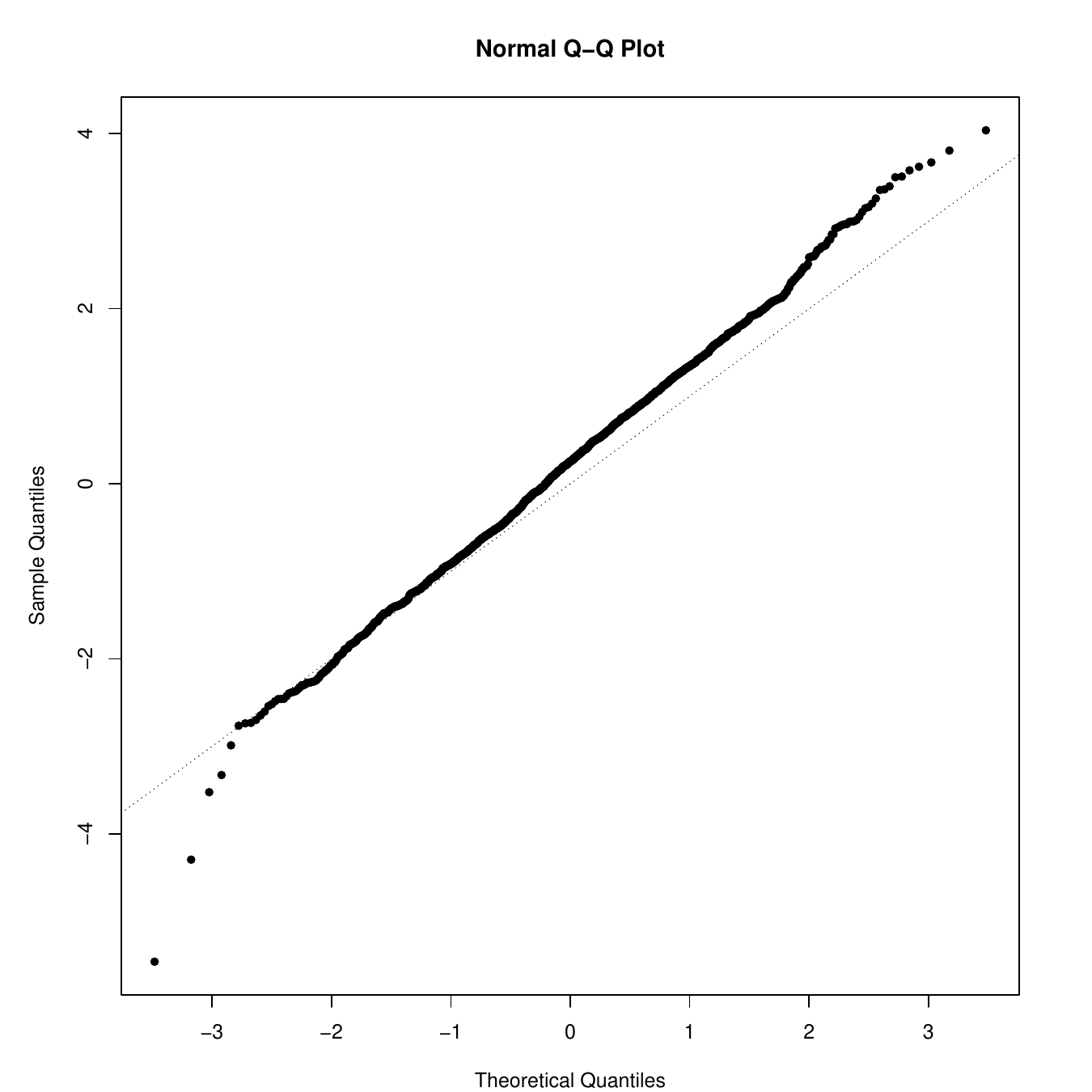}
\includegraphics[width=0.48\linewidth]{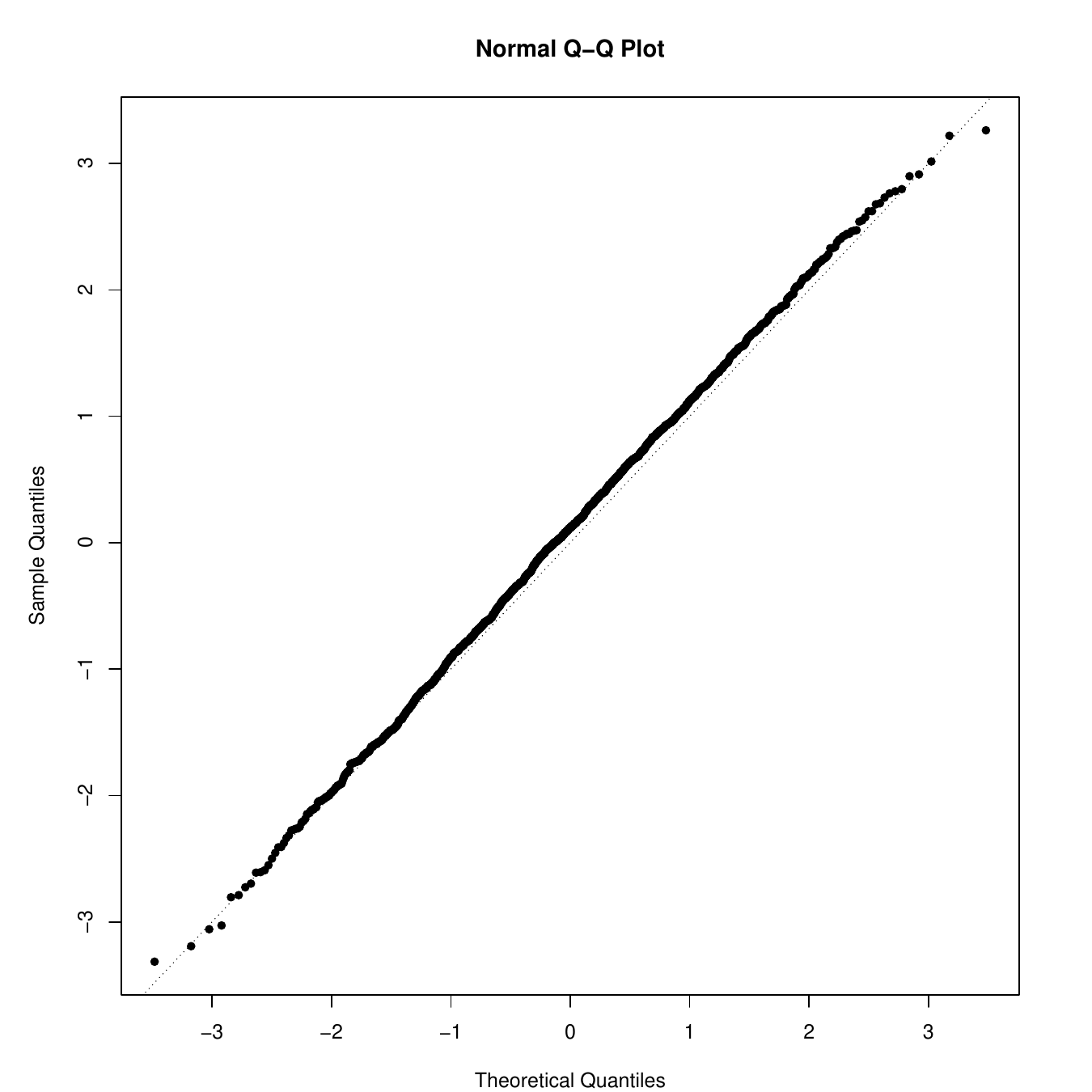}
        \caption{QQ plot for $T_N$ for $N=50$ (left panel) and $N=500$ (right panel) against the standard Normal distribution.}
        \label{fig:QQ500-both}
    \end{figure}
    For $N=500$, we note that the empirical variance of $\hat{\theta}_n$ multiplied by $N$ equals $6.82$.  The average of the variance estimates (as described in Section \ref{varestss}) is equal to $6.56$. The theoretical asymptotic variance, equal to the variance bound for this estimator, is $1/\mathbb{E}\left[e_0(X)\left\{1-e_0(X)\right\}\right]$, which in this case can be estimated to be $5.13$ using Monte Carlo.

\begin{table}[ht]
    \centering
    \begin{tabular}{lcccc}
    $N$ & 50 & 100 & 500 & 1000\\
    \hline
    \\[-6pt]
    Empirical variance&9.46&8.01&6.82&6.39\\
    \\[-6pt] 
    Average $\hat \Sigma$&7.74&7.42&6.56&6.31\\
    \\[-6pt]
    Average posterior variance & 7.20&7.05&6.47&6.27\\
    \\[-6pt]
    \hline
    \\[-6pt]
    \end{tabular}
    \caption{Frequentist asymptotic variances and $N$ times the average posterior variance for different sample sizes.}
    \label{tab:FreqVar}
\end{table}
    
    \item The posterior (conditional) distribution distribution of $T_{N,BB}$ given $O_1,\ldots,O_N$ is standard Normal, mimicking the (unconditional) frequentist distribution. We ran a Shapiro-Wilk test at a significance level of 0.05 across the 2000 posteriors obtained. The rejection rates for the (null) hypothesis that the posterior distribution of $T_{N,BB}$ is standard Normal were $0.275, 0.141,0.064$ and $0.056$ for the four sample sizes. The values of $N$ times the average posterior variances are displayed in the final row of Table \ref{tab:FreqVar}.

  %\item We also add the plots for the smaller sample-size of $N=250$, see Figures \ref{N=250Freq} and \ref{N=250Post}. We note in this case that the empirical variance of $\hat{\theta}_n$ times $N$ is equal to $7.79$ and that the empirical bias of $\hat{\theta}_n$ is $0.0193$. The average of the variance estimates as described in Section \ref{varestss} is $7.55$ and $N$ times the average of the $1000$ empirical posterior variances is $7.34$.

\item We may assess (frequentist) coverage of uncertainty intervals based on $\hat \theta_n$ and $\hat \theta_{n,BB}$ using the asymptotic Normal approximation and credible interval respectively.  The coverage values at nominal level 0.95 are presented in Table \ref{tab:Cover}.

\begin{table}[ht]
    \centering
    \begin{tabular}{lcccc}
    $N$ & 50 & 100 & 500 & 1000\\
    \hline
    \\[-6pt]
   Frequentist&1.000&1.000&1.000&1.000\\
    \\[-6pt] 
    Bayesian & 0.889& 0.918&0.940& 0.945\\
    \\[-6pt]
    \hline
    \\[-6pt]
    \end{tabular}
    \caption{Empirical coverage of frequentist and Bayesian uncertainty intervals}
    \label{tab:Cover}
\end{table}

The coverage of the frequentist interval is above the nominal level due to the variance estimator from Section \ref{varestss} being too large on average.  The coverage of the Bayesian credible interval is below the nominal level when $N$ is small, but reasonable when $N=500$.
\end{enumerate}

The findings underline the importance of Neyman orthogonality in this Bayesian context in one notable aspect; the nuisance parameters can be estimated and remain unchanged during the Bayesian procedure, without affecting the posterior variance. That is, Neyman orthogonality allows us to disregard the uncertainty attached to estimating the nuisance parameters, and facilitate the straigtforward calculation of $\hat \theta_{n,BB}$ in \eqref{eq:DRG} as the two residual terms in numerator and denominator
\[
(Z_i - \hat e_0(X_i)) \qquad (Y_i - \hat k_y^0(X_i))
\]
are fixed across the Bayesian bootstrap resamples.

\section{Conclusion}
\label{sec:conc}
We have shown how to sample from a posterior distribution of a targeted parameter of interest in presence of a highly complex nuisance parameter. We have avoided restrictions on the nuisance space by the use of cross-fitting as in \citet{Chernozhukov/etal:2018}.  Our results establish the validity of the bootstrap calculation for sample-splitting algorithms. Technical details as well as possible extensions to non-differentiable scores and modified conditions for some theorems are found in \citet{Sabbagh2025}.
%, PhD thesis ``Bayesian Causal Inference in Semi-Parametric Models''.
Our results highlight the importance of three elements; the central concept of of Neyman orthogonality and its impact on estimation; the need for estimation of $h_0$ at an acceptable rate, and the utility of sample-splitting when Donsker conditions are relaxed.  Using sample-splitting, reasonable inferential performance can be obtained in fairly small samples even when the nuisance component is estimated non-parametrically and the predictors are relatively high-dimensional.  

We have focussed on the Bayesian bootstrap due to its fully Bayesian nature, that is, it is a calculation that arises from a Bayesian non-parametric posterior.  In parallel, it is natural to consider the classical (Efron) non-parametric bootstrap as a competing method.  The classical bootstrap method is covered by the theoretical results we present; the random weights arise from a standardized draw from a multinomial distribution on $\{1,2,\ldots,n\}$ with equal probability $1/n$ for each category instead of the Dirichlet distribution with all parameters equal to 1.  Along with a lack of fully coherent Bayesian justification for the classical bootstrap, there are other key differences, most notably the (marginal) variance of the weights is larger for the classical version by a factor $1/n$ -- that is, for $i=1,\ldots,n$,
\[
\text{Var}_{\text{Efron}}[w_{in}] = \left( 1 + \frac{1}{n}\right) \text{Var}_{BB}[w_{in}]
\]
which ultimately leads to a higher variance of the classical bootstrap estimator.  When $N=50$, $N$ times the average `posterior' variance is equal to 8.48, compared to 7.20 for the Bayesian bootstrap from Table \ref{tab:FreqVar}; when $N=1000$, the equivalent values are 6.33 and 6.31.  For $N=50$, there is a consequent increase in coverage to 0.921 (from 0.889).  The coverage calculation here is the frequentist version (with a single fixed parameter rather than a parameter sampled from the prior).  For a single data set, the Bayesian bootstrap provides a posterior for $\theta_0$ that has on average a smaller variance than that computed by the classical bootstrap, a result observed in parametric Bayesian settings.  Similar calculations can be performed for other types of bootstrap as in \citet{Praestgaard/Wellner:1993}.

\appendix

\section{Technical Details}

 \begin{App_prop}
     Let $\mathbf{W}_n=(w_{1n},\ldots,w_{nn})$ be a random vector of non-negative components summing to $1$ such that $\mathbb{E}w_{in}=1/{n}$ and $\sup\limits_{n \in \mathbb{N} } n^2 \mathbb{E}w_{in}^2 < C <\infty$.
     Then under assumption
     \[
  \sum\limits_{i=1}^n w_{in}A(O_i;\hat{h}_2) \quad \text{ and } \quad \frac{1}{n}\sum\limits_{i=1}^nA(O_i,\hat{h}_2) 
  \]
  converge unconditionally to $\mathbb{E}A(O;h_0)$.
 \end{App_prop}
 \begin{proof}
 We prove the second statement and the first would follow by setting the weights to be equal to $n^{-1}$.
    \begin{align*}
\sum\limits_{i=1}^n w_{in} A(O_i;\hat{h}_2) - \mathbb{E} A(O,h_0) & =
\sum\limits_{i=1}^n w_{in} A(O_i;\hat{h}_2)- \mathbb{E}_{ I_{1}|I_2,\mathbf{W}_n} \sum\limits_{i=1}^n w_{in} A(O_i;\hat{h}_2)\\
& \quad +\mathbb{E}_{ I_{1}|I_2,\mathbf{W}_n} \sum\limits_{i=1}^n w_{in} A(O_i;\hat{h}_2)- \mathbb{E}A(O;h_0).
\end{align*}

Let 
\[
E_1=\sum\limits_{i=1}^n w_{in} A(O_i;\hat{h}_2)- \mathbb{E}_{ I_{1}|I_2,\mathbf{W}_n}\left\{ \sum\limits_{i=1}^n w_{in} A(O_i;\hat{h}_2)|I_2,\mathbf{W}_n\right\}
\]
and 
\[
E_2= \mathbb{E}_{ I_{1}|I_2,\mathbf{W}_n}\left\{ \sum\limits_{i=1}^n w_{in} A(O_i;\hat{h}_2)|I_2,\mathbf{W}_n\right\} -\mathbb{E}A(O;h_0).
\]
The law of total expectation implies that $\mathbb{E} [E_1] =0_{p \times p}.$

Moreover, if $j,k \in \left\{1,\ldots,p\right\},$
\begin{align*}
\var \left\{[E_1]_{jk} | I_2,\mathbf{W}_n\right\} &= \var \left\{ \sum\limits_{i=1}^n w_{in} A_{jk}(O_i;\hat{h}_2)| I_2,\mathbf{W}_n \right\} \\
&= \sum\limits_{i=1}^n w_{in}^2 \var \left\{A_{jk}(O_i;\hat{h}_2)| I_2\right\}\\
& \leq \sum\limits_{i=1}^n  w_{in}^2 \sup\limits_{h } \var  A_{jk}(O_i;h) \\
& \leq C' \sum\limits_{i=1}^n w_{in}^2
\end{align*}
Moreover, since $\mathbb{E}[E_1|I_2,\mathbf{W}_n]=0_{p \times p}$, then the law of total variance implies that 
$$
\var [E_1]_{jk}= \mathbb{E}_{I_2,\mathbf{W}_n} \var \left\{[E_1]_{jk}|I_2,\mathbf{W}_n\right\} \leq C' \mathbb{E}_{\mathbf{W}_n} \sum\limits_{i=1}^n w_{in}^2 \leq \frac{CC'}{n}.$$
Chebychev's inequality hence implies that $E_1$ converges to $0_{p \times p}$.
Now, note that 
\begin{align*}
E_2 &= \mathbb{E}_{ I_{1}|I_2,\mathbf{W}_n}\left\{ \sum\limits_{i=1}^n w_{in} A(O_i;\hat{h}_2)|I_2,\mathbf{W}_n\right\} -\mathbb{E}A(O;h_0)\\
&= \mathbb{E}_{ I_{1}|I_2}\left\{ \frac{1}{n}\sum\limits_{i=1}^n  A(O_i;\hat{h}_2)|I_2,\right\} -\mathbb{E}A(O;h_0)\\
&=\mathbb{E}_{O_1|I_2} \left\{ \  A(O_1;\hat{h}_2)|I_2\right\} -\mathbb{E} A(O;h_0)
\end{align*}
\begin{align*}
\left\vert [E_2]_{jk} \right\vert &= \left\vert \mathbb{E}_{O_1|I_2} \left\{ \  A_{jk}(O_1;\hat{h}_2)|I_2\right\} -\mathbb{E} A_{jk}(O;h_0) \right\vert \\
& \leq \sup\limits_{h \in H_N} \left\vert \mathbb{E}_{O_1|I_2} \left\{ \  A_{jk}(O_1;h)|I_2\right\} -\mathbb{E} A_{jk}(O;h_0) \right\vert\\
& \leq a_N
\end{align*}
This hence implies that $E_2$ converges to $0_{p \times p}$, establishing the previous two statements.
\end{proof}
\begin{App_prop}
Let $\mathbf{W}_n=(w_{1n},\ldots,w_{nn})$ be a random vector of non-negative components summing to $1$ such that $\mathbb{E}w_{in}=1/{n}$ and $\sup\limits_{n \in \mathbb{N} } n^2 \mathbb{E}w_{in}^2 < C <\infty$. Then, under Assumption \ref{AssonVarm},
 
\[
\sqrt{n}\sum\limits_{i=1}^n \left(w_{in}-\frac{1}{n}\right) \left\{ m\left(O_i;\theta_0,\hat{h}_2\right) -m\left(O_i;\theta_0,h_0\right)\right\}
\]
converges to $0$ in probability.
\end{App_prop}
\begin{proof}
We first show that the expected value of this term is exactly $0$ for each $n \in \mathbb{N}$.
\begin{align*}
\mathbb{E}_{O_{1}:O_{2n},\mathbf{W}_n} \sum\limits_{i=1}^n \left(w_{in}-\frac{1}{n}\right) & \left\{ m\left(O_i;\theta_0,\hat{h}_2\right) -m\left(O_i;\theta_0,h_0\right)\right\}\\
& = \sum\limits_{i=1}^n \mathbb{E}_{O_{1}:O_{2n},\mathbf{W}_n} \left[\left(w_{in}-\frac{1}{n}\right) \left\{ m\left(O_i;\theta_0,\hat{h}_2\right) -m\left(O_i;\theta_0,h_0\right)\right\}\right]\\
& =\sum\limits_{i=1}^n \mathbb{E}_{O_{1}:O_{2n},\mathbf{W}_n} \left[\left(w_{in}-\frac{1}{n}\right) \left\{ m\left(O_i;\theta_0,\hat{h}_2\right) -m\left(O_i;\theta_0,h_0\right)\right\}\right]\\
& = \sum\limits_{i=1}^n \mathbb{E}_{\mathbf{W}_n}\left(w_{in}-\frac{1}{n}\right) \mathbb{E}_{O_1:O_{2n}}\left\{ m\left(O_i;\theta_0,\hat{h}_2\right) -m\left(O_i;\theta_0,h_0\right)\right\}\\
& =0
\end{align*}

We now show that the variance 
\[
\text{Var}_{O_1:O_{2n},\mathbf{W}_n} \sum\limits_{i=1}^n 
 \sqrt{n}\left[\left(w_{in}-\frac{1}{n}\right) \left\{ m\left(O_i;\theta_0,\hat{h}_2\right) -m\left(O_i;\theta_0,h_0\right)\right\}\right]
\]
goes to $0$. We use the total variance formula, by conditioning on $\mathbf{W}_n$ and on $I_2$. Then
\[
\text{Var}_{O_1:O_{2n},\mathbf{W}_n} \sum\limits_{i=1}^n 
 \sqrt{n}\left[\left(w_{in}-\frac{1}{n}\right) \left\{ m\left(O_i;\theta_0,\hat{h}_2\right) -m\left(O_i;\theta_0,h_0\right)\right\}\right]=n(A+B),
\]
where
\begin{align*}
A& =\mathbb{E}_{I_2,\mathbf{W}_n} \left[\text{Var}_{I_1|I_2,\mathbf{W}_n} \sum\limits_{i=1}^n \left[\left(w_{in}-\frac{1}{n}\right) \left(m\left(O_i;\theta_0,\hat{h}_2\right) -m\left(O_i;\theta_0,h_0\right)\right)\right]|I_2,\mathbf{W}_n\right]\\[6pt]
B& =\text{Var}_{I_2,\mathbf{W}_n} \left[\mathbb{E}_{I_1|I_2,\mathbf{W}_n} \sum\limits_{i=1}^n \left[\left(w_{in}-\frac{1}{n}\right) \left(m\left(O_i;\theta_0,\hat{h}_2\right) -m\left(O_i;\theta_0,h_0\right)\right)\right]|I_2,\mathbf{W}_n\right]
\end{align*}
We show that $nA$ goes to $0$  goes to $0$ as $N \to \infty$. 
\begin{align*}
A&=\mathbb{E}_{I_2,\mathbf{W}_n} \left[\text{Var}_{I_1|I_2,\mathbf{W}_n} \sum\limits_{i=1}^n \left[\left(w_{in}-\frac{1}{n}\right) \left(m\left(O_i;\theta_0,\hat{h}_2\right) -m\left(O_i;\theta_0,h_0\right)\right)\right]|I_2,\mathbf{W}_n\right]\\
&=\mathbb{E}_{I_2,\mathbf{W}_n}\sum\limits_{i=1}^n (w_{in}-\frac{1}{n})^2\text{Var}_{I_1|I_2,\mathbf{W}_n} \left\{m\left(O_i;\theta_0,\hat{h}_2\right) -m\left(O_i;\theta_0,h_0\right)\right\}|I_2,\mathbf{W}_n\\
&=\mathbb{E}_{I_2,\mathbf{W}_n}\sum\limits_{i=1}^n (w_{in}-\frac{1}{n})^2\text{Var}_{I_1|I_2} \left\{m\left(O_i;\theta_0,\hat{h}_2\right) -m\left(O_i;\theta_0,h_0\right)\right\}|I_2\\
&=\mathbb{E}_{I_2,\mathbf{W}_n}\sum\limits_{i=1}^n (w_{in}-\frac{1}{n})^2\left[\max_{j=1}^p\text{Var}_{I_1|I_2} \left\{m_j\left(O_i;\theta_0,\hat{h}_2\right) -m_j\left(O_i;\theta_0,h_0\right)\right\}|I_2 \right] \mathbf{1}_{p}\mathbf{1}_p^\top\\
&\leq \mathbb{E}_{I_2,\mathbf{W}_n}\sum\limits_{i=1}^n (w_{in}-\frac{1}{n})^2\left[ \max_{j=1}^p\text{Var}_{I_1|I_2} \left\{m_j\left(O_i;\theta_0,h\right) -m_j\left(O_i;\theta_0,h_0\right)\right\}|I_2\right] \mathbf{1}_{p}\mathbf{1}_{p}^\top \\
%&\leq \mathbb{E}_{I_2,\mathbf{W}_n}\sum\limits_{i=1}^n (w_{in}-\frac{1}{n})^2 \left[\sup_{h \in H} \max_{j=1}^p\text{Var}_{I_1} \left\{m_j\left(O_i,\theta_0,h\right) -m_j\left(O_i,\theta_0,h_0\right)\right\}\right] \mathbf{1}_{p}\mathbf{1}_{p}^T\\
%&\leq \mathbb{E}_{I_2,\mathbf{W}_n}\sum\limits_{i=1}^n (w_{in}-\frac{1}{n})^2 \delta_N^2 \mathbf{1}_{p}\mathbf{1}_{p}^T\\
&\leq \delta_N^2 \sum\limits_{i=1}^n \mathbb{E}_{\mathbf{W}_n}  (w_{in}-\frac{1}{n})^2 \mathbf{1}_{p}\mathbf{1}_{p}^\top\\[6pt]
&\leq \delta_N^2 \sum\limits_{i=1}^n  \text{Var}_{\mathbf{W}_n}w_{in}\mathbf{1}_{p}\mathbf{1}_{p}^\top\\[6pt]
& \leq \frac{C \delta_N^2}{n} \mathbf{1}_{p}\mathbf{1}_{p}^\top
\end{align*}
where we have used the conditional Cauchy-Schwarz inequality to go from the third to the fourth line in bounding $A$. Indeed, let $k \in \left\{1,\ldots,n \right\}$. For any $i,j \in \left\{1,\ldots,p\right\},$
\begin{align*}
&\left[\text{Var}_{I_1|I_2} \left\{m\left(O_k;\theta_0,\hat{h}_2\right) -m\left(O_k;\theta_0,h_0\right)\right\}|I_2 \right]_{ij}\\[6pt]
&=\text{Cov}_{I_1|I_2} \left[\left\{m_i\left(O_k;\theta_0,\hat{h}_2\right) -m_i\left(O_k;\theta_0,h_0\right)\right\},\left\{m_j\left(O_k;\theta_0,\hat{h}_2\right) -m_j\left(O_k;\theta_0,h_0\right)\right\}|I_2 \right] \\[6pt]
&\leq \sqrt{ \text{Var}_{I_1|I_2} \left\{ m_i\left(O_k;\theta_0,\hat{h}_2\right) -m_i\left(O_k;\theta_0,h_0\right)\right\}|I_2} \sqrt{ \text{Var}_{I_1|I_2} \left\{m_j\left(O_k;\theta_0,\hat{h}_2\right) -m_j\left(O_k;\theta_0,h_0\right)\right\}|I_2}\\[6pt]
& \leq \max_{1 \le i \le p }  \text{Var}_{I_1|I_2} \left\{ m_i\left(O_k;\theta_0,\hat{h}_2\right) -m_i\left(O_k;\theta_0,h_0\right)\right\}|I_2\\[6pt]
& \leq \sup\limits_{h \in \mathcal{H}} \max_{1 \le i \le p }  \text{Var}_{I_1|I_2} \left\{ m_i\left(O_k;\theta_0,h\right) -m_i\left(O_k;\theta_0,h_0\right)\right\}|I_2\\[6pt]
&\leq \sup\limits_{h \in \mathcal{H}} \max_{1 \le i \le p }  \text{Var}_{I_1|I_2} \left\{ m_i\left(O_k;\theta_0,h\right) -m_i\left(O_k;\theta_0,h_0\right)\right\}\\[6pt]
& \le \delta_N^2
\end{align*}
Since $A \geq 0$, it follows that $nA \to 0$.

We now look at $B$.
\begin{align*}
B&=\text{Var}_{I_2,\mathbf{W}_n} \left[\mathbb{E}_{I_1|I_2,\mathbf{W}_n} \sum\limits_{i=1}^n \left[\left(w_{in}-\frac{1}{n}\right) \left(m\left(O_i;\theta_0,\hat{h}_2\right) -m\left(O_i;\theta_0,h_0\right)\right)\right]|I_2,\mathbf{W}_n\right]\\
&=\text{Var}_{I_2,\mathbf{W}_n}\sum\limits_{i=1}^n \left[\mathbb{E}_{I_1|I_2,\mathbf{W}_n}\left[\left(w_{in}-\frac{1}{n}\right) \left(m\left(O_i;\theta_0,\hat{h}_2\right) -m\left(O_i;\theta_0,h_0\right)\right)\right]|I_2,\mathbf{W}_n\right]\\
&=\text{Var}_{I_2,\mathbf{W}_n}\sum\limits_{i=1}^n \left(w_{in}-\frac{1}{n}\right) \left[\mathbb{E}_{I_1|I_2,\mathbf{W}_n}\left[\left(m\left(O_i;\theta_0,\hat{h}_2\right) -m\left(O_i;\theta_0,h_0\right)\right)\right]|I_2,\mathbf{W}_n\right]\\
&=\text{Var}_{I_2,\mathbf{W}_n}\sum\limits_{i=1}^n \left(w_{in}-\frac{1}{n}\right) \left[\mathbb{E}_{I_1|I_2}\left[\left(m\left(O_i;\theta_0,\hat{h}_2\right) -m\left(O_i;\theta_0,h_0\right)\right)\right]|I_2\right]\\
&=\text{Var}_{I_2,\mathbf{W}_n}\sum\limits_{i=1}^n \left(w_{in}-\frac{1}{n}\right) \left[\mathbb{E}_{I_1|I_2}\left[\left(m\left(O_1;\theta_0,\hat{h}_2\right) -m\left(O_1;\theta_0,h_0\right)\right)\right]|I_2\right]\\
&=\text{Var}_{I_2,\mathbf{W}_n}\sum\limits_{i=1}^n \left(w_{in}-\frac{1}{n}\right) \left[\mathbb{E}_{O_1|I_2}\left[\left(m\left(O_1;\theta_0,\hat{h}_2\right) -m\left(O_1;\theta_0,h_0\right)\right)\right]|I_2\right]\\
&=\text{Var}_{I_2,\mathbf{W}_n} \left[\mathbb{E}_{O_1|I_2}\left[\left(m\left(O_1;\theta_0,\hat{h}_2\right) -m\left(O_1;\theta_0,h_0\right)\right)\right]|I_2\right]\sum\limits_{i=1}^n \left(w_{in}-\frac{1}{n}\right) \\
&=0_{p \times p}
\end{align*}
Therefore, by Chebychev's inequality, we get the desired result. 
\end{proof}
\begin{proposition}
Let $\mathbf{W}_n=(w_{1n},\ldots,w_{nn})$ be a random vector of non-negative components summing to $1$ such that $\mathbb{E}w_{in}=1/{n}$ and $\sup\limits_{n \in \mathbb{N} } n^2 \mathbb{E}w_{in}^2 < C <\infty$. Then, under Assumption \ref{AssonVarm},
 
 \[
  \sqrt{n} \left[\frac{1}{n}\sum\limits_{i=1}^n\left\{m(O_i;\theta_0,\hat{h}_2)-m(O_i;\theta_0,h_0)\right\}-\mathbb{E}_{P_O}\left\{m\left(O;\theta_0,\hat{h}_2\right)-m(O;\theta_0,h_0)\right\}\right]
  \]
converges to $0$ in probability.
\end{proposition}
\begin{proof}
As in the previous proposition, we look at the expected value and variance of that term.
\begin{align*}
&\mathbb{E}_{O_1:O_{2n}} \left[\frac{1}{n}\sum\limits_{i=1}^n\left\{m(O_i;\theta_0,\hat{h}_2)-m(O_i;\theta_0,h_0)\right\}-\mathbb{E}_{P_O}\left\{m\left(O;\theta_0,\hat{h}_2\right)-m(O;\theta_0,h_0)\right\}\right]\\
&=\mathbb{E}_{O_{1}:O_{2n}}\left[\frac{1}{n}\sum\limits_{i=1}^n\left\{m(O_i;\theta_0,\hat{h}_2)-m(O_i;\theta_0,h_0)\right\}\right]-\mathbb{E}_{P_O,I_2}\left\{m\left(O;\theta_0,\hat{h}_2\right)-m(O;\theta_0,h_0)\right\}\\
&=\mathbb{E}_{I_2}\mathbb{E}_{I_1|I_2}\left[\frac{1}{n}\sum\limits_{i=1}^n\left\{m(O_i;\theta_0,\hat{h}_2)-m(O_i;\theta_0,h_0)\right\}|I_2\right]-\mathbb{E}_{P_O}\left\{m\left(O;\theta_0,\hat{h}_2\right)-m(O;\theta_0,h_0)\right\}\\
&=\mathbb{E}_{I_2} \mathbb{E}_{I_1|I_2} m(O_1;\theta_0,\hat{h}_2)|I_2 -\mathbb{E}_{P_{O},I_2}\left\{m\left(O;\theta_0,\hat{h}_2\right)-m(O;\theta_0,h_0)\right\}\\[6pt]
&=0.
\end{align*}
We now look at the following variance and use the total variance formula:\\
\[
\text{Var}_{O_1: O_{2n}} \sqrt{n}\left[\frac{1}{n}\sum\limits_{i=1}^n\left\{m(O_i;\theta_0,\hat{h}_2)-m(O_i;\theta_0,h_0)\right\}-\mathbb{E}_{P_O}\left\{m\left(O;\theta_0,\hat{h}_2\right)-m(O;\theta_0,h_0)\right\}\right]=n(C+D)
\]
where
\[
C= \mathbb{E}_{I_2} \text{Var}_{I_1|I_2} \left[ \frac{1}{n}\sum\limits_{i=1}^n\left\{m(O_i;\theta_0,\hat{h}_2)-m(O_i;\theta_0,h_0)\right\}-\mathbb{E}_{P_O}\left\{m\left(O;\theta_0,\hat{h}_2\right)-m(O;\theta_0,h_0)\right\}|I_2\right],
\]
and
\[
D=\text{Var}_{I_2}\mathbb{E}_{I_1|I_2} \left[ \frac{1}{n}\sum\limits_{i=1}^n\left\{m(O_i;\theta_0,\hat{h}_2)-m(O_i;\theta_0,h_0)\right\}-\mathbb{E}_{P_O}\left\{m\left(O;\theta_0,\hat{h}_2\right)-m(O;\theta_0,h_0)\right\}|I_2\right].
\]
Now
\begin{align*}
C&= \mathbb{E}_{I_2} \text{Var}_{I_1|I_2} \left[ \frac{1}{n}\sum\limits_{i=1}^n\left\{m(O_i;\theta_0,\hat{h}_2)-m(O_i;\theta_0,h_0)\right\}-\mathbb{E}_{P_O}\left\{m\left(O;\theta_0,\hat{h}_2\right)-m(O;\theta_0,h_0)\right\}|I_2\right] \\
&=\mathbb{E}_{I_2} \text{Var}_{I_1|I_2}\left[ \frac{1}{n}\sum\limits_{i=1}^n\left\{m(O_i;\theta_0,\hat{h}_2)-m(O_i;\theta_0,h_0)\right\}|I_2\right]\\
&=\frac{1}{n}\mathbb{E}_{I_2}\text{Var}_{I_1|I_2} \left[m(O_1;\theta_0,\hat{h}_2)-m(O_1;\theta_0,h_0)|I_2\right]\\[6pt]
&\leq \frac{\delta^2_N}{n}\mathbf{1}_p\mathbf{1}_p^\top
\end{align*}

and 
\begin{align*}
D&=\text{Var}_{I_2}\mathbb{E}_{I_1|I_2} \left[ \frac{1}{n}\sum\limits_{i=1}^n\left\{m(O_i;\theta_0,\hat{h}_2)-m(O_i;\theta_0,h_0)\right\}-\mathbb{E}_{P_O}\left\{m\left(O;\theta_0,\hat{h}_2\right)-m(O;\theta_0,h_0)\right\}|I_2\right]\\
&=\text{Var}_{I_2}\mathbb{E}_{I_1|I_2} \left[ \frac{1}{n}\sum\limits_{i=1}^n\left\{m(O_i;\theta_0,\hat{h}_2)\right\}-\mathbb{E}_{P_O}\left\{m\left(O;\theta_0,\hat{h}_2\right)\right\}|I_2\right]\\
&=\text{Var}_{I_2}\left[\left\{\mathbb{E}_{I_1|I_2} m(O_1;\theta_0,\hat{h}_2)|I_2\right\}-\mathbb{E}_{O_1|I_2} m(O_1;\theta_0,\hat{h}_2)|I_2\right]\\[6pt]
&=0_{p \times p}
\end{align*}
which then implies the desired result that 
  \[
  \sqrt{n}\left[ \frac{1}{n}\sum\limits_{i=1}^n\left\{m(O_i;\theta_0,\hat{h}_2)-m(O_i;\theta_0,h_0)\right\}-\mathbb{E}_{P_O}\left\{m\left(O;\theta_0,\hat{h}_2\right)-m(O;\theta_0,h_0)\right\}\right]
  \]
 converges to $0$ in probability.\\
 \end{proof}

\section{Sample-splitting with replicate splits: reporting the posterior} \label{sec:RepSplits}
There are a number of ways the sample-splitting algorithm can be deployed to report the best representation of the posterior distribution.  Algorithm \ref{BBSampleSplittingIndepWeights} provides one version that relies upon averaging across splits. Since the fits are based on subsamples obtained by randomly dividing the data, we may also follow the recommendation of \citet{Chernozhukov/etal:2018} to minimize the effect of a particular random split on the estimation. We repeat the estimation for an odd number of times $n_{\text{reps}}$, to obtain estimates $\theta_1,\ldots,\theta_{n_{\text{reps}}}$ of the parameter of interest, which is a scalar in our case. We then report the median $\theta_{\text{med}}$ and plot the posterior for this particular choice of subsamples. Indeed, although the choice of the splits is not significant asymptotically, the procedure discussed in \citet{Chernozhukov/etal:2018} and in Algorithm \ref{FiniteSamplePosteriorSS} reduces the effect of a particular random split in finite samples. Concretely, we implement the following algorithm: 
\begin{algo}{\label{FiniteSamplePosteriorSS}}

Sample-splitting with replicated splits:

    \begin{enumerate}
        \item Choose an odd number $n_{\text{reps}}$ and a number $n_{\text{samp}}$ of posterior samples.
        \item For each $i$ between $1$ and $n_{\text{reps}}$:
              \begin{enumerate}
                  \item Split the data into two subsamples $I_1^{(i)}$ and $I_2^{(i)}$ of the same size.
                  \item For each $j$ from $1$ to $2$,
                        \begin{enumerate}
                            \item Exclude $I_j^{(i)}$ and obtain $\hat{h}_{-j}$ on the remaining  observations.
                            \item Test the models on $I_j^{(i)}$ and get the fitted values $\left\{\hat{h}_{-j}\left(O_s\right)\right\}_{s \in I_j^{(i)}}$.
                            \item Compute $\displaystyle{\hat{\theta}^{(i)}_{j}=-\frac{\sum\limits_{i \in I_j}B(O_i,\hat{h}_{-j}) }{\sum\limits_{i \in I_j}A(O_i,\hat{h}_{-j})}}$ 
                        \end{enumerate}
                  \item Compute $\hat{\theta}_{i}=(\hat{\theta}_1^{(i)}+\hat{\theta}^{(i)}_2)/2$
              \end{enumerate}
        \item Identify $i_{\text{med}}$ such that the median of $\hat{\theta}_1,\ldots,\hat{\theta}_{n_{\text{reps}}}$ is equal to $\hat{\theta}_{i_{\text{med}}}$. 
        \item Report the posterior for $\theta$ based on the subsamples $I_1^{(i_{\text{med}})}$ and  $I_2^{(i_{\text{med}})}$ using $n_{\text{samp}}$ Bayesian bootstrap samples.
    \end{enumerate}
\end{algo}

\bibliographystyle{biometrika}
\bibliography{references}
\end{document}